\documentclass{article}
\usepackage{amsthm}
\usepackage{amsmath,amssymb}
\usepackage{helvet}
\usepackage{docmute}
\usepackage{tcolorbox}
\usepackage[enableskew]{youngtab}
\usepackage{pxrubrica}
\usepackage{xcolor}
\usepackage{tikz}
\usepackage{ytableau}
\usepackage{braket}
\usepackage{autobreak}
\usepackage{mathtools}
\usepackage{hyperref}
\usepackage{endnotes}
\usepackage{authblk}
\hypersetup{%
 setpagesize=false,%
 bookmarks=true,%
 bookmarksdepth=tocdepth,%
 bookmarksnumbered=true,%
 colorlinks=false,%
 pdftitle={},%
 pdfsubject={},%
 pdfauthor={},%
 pdfkeywords={}}
\theoremstyle{definition}
\newtheorem{dfn}{Definition}[section]
\newtheorem{thm}{Theorem}[section]
\newtheorem{lem}{Lemma}[section]

\newtheorem{prop}{Proposition}[section]
\newtheorem{cor}{Corollary}[section]
\newtheorem*{main*}{Main result}
\theoremstyle{definition}
\newtheorem{case}{Example}[section]
\newtheorem*{prop*}{proposition}
\numberwithin{equation}{section}

\newtheorem{rem}{Remark}[section]
\theoremstyle{plain}

\makeatletter
\renewcommand{\AB@affilsep}{\quad\protect\Affilfont}
\let\AB@affilsepx\AB@affilsep 

\makeatother

\begin{document}
\title{Special values of $K$-theoretic Schur $P$- and $Q$-functions}
\author{Takahiko Nobukawa and Tatsushi Shimazaki}
\date{}
\maketitle

\begin{abstract}
We provide the special values of the skew version of the $K$-theoretic Schur $P$- and $Q$-functions. Using these special values, we show an oddness property of the number of shifted set-valued skew tableaux. Additionally, we generalize these special values to another skew case. 
Based on these special values, we give pairs among certain shifted set-valued skew tableaux.
\end{abstract}
\renewcommand{\thefootnote}{\fnsymbol{footnote}}
\renewcommand{\thefootnote}{\arabic{footnote}}

\section{Introduction}\label{section1}

The Schur polynomial is the character of finite-dimensional irreducible representations of the general linear group. 
As is well known, the Schur polynomial can be represented in terms of semistandard Young tableaux.
Throughout this paper, we use the term semistandard tableaux as an abbreviation for semistandard Young tableaux.
The Grothendieck polynomial is known as a $K$-theoretic analogue of the Schur polynomial. 
The set-valued semistandard tableaux, which are generalizations of semistandard tableaux, were introduced by Buch~\cite{Buc02} to represent the (stable) Grothendieck polynomial. The Grothendieck polynomial was introduced by Lascoux and Schützenberger~\cite{Las90,LS82} to serve as representatives for the structure sheaves associated with Schubert varieties in flag varieties.

In addition to the Schur polynomial, the Schur $P$- and $Q$-functions appear in the projective representations of the symmetric group~\cite{Sch11}. 
These functions can be expressed as polynomials using shifted semistandard tableaux. 
Schur $P$-functions can be expressed as scalar multiples of Schur $Q$-functions (see~\eqref{P=Q}).
Ivanov~\cite{Iva01} introduced a multiparameter deformation of Schur $Q$-functions and proved several combinatorial formulas analogous to those for the original Schur $Q$-functions.

Ikeda and Naruse \cite{IN13} introduced a $K$-theoretic analogue of Ivanov's functions, called the $K$-theoretic factorial Schur $P$- and $Q$-functions. 
They proved that these functions correspond to the structure sheaves of Schubert varieties in the torus-equivariant $K$-theory of maximal isotropic Grassmannians of symplectic or orthogonal types. 
Since the Grothendieck polynomial is a $K$-theoretic analogue of the Schur function, these $K$-theoretic Schur $P$- and $Q$-functions can be regarded as analogues in the shifted setting. 
In this paper, we adopt the definition of the non-factorial $K$-theoretic Schur $P$- and $Q$-functions using shifted set-valued semistandard tableaux. 
Our main result is as follows:
\begin{main*}[Theorem~\ref{skewspecialvalue}]
Let $\lambda$ and $\mu$ be strict partitions such that $\lambda \supset \mu$.
We have
\begin{align*}
GP_{\lambda/\mu}(\beta,\dots,\beta \mid -\beta^{-1}) &= \beta^{|\lambda/\mu|}, \\
GQ_{\lambda/\mu}(\beta,\dots,\beta \mid -\beta^{-1}) &= \beta^{|\lambda/\mu|}.
\end{align*}
\end{main*}
As a corollary of this result, we show that the number of shifted set-valued skew tableaux is always odd whenever it is non-zero.
Analogous results for the Grothendieck polynomial, its skew version, and the number of ordinary semistandard skew tableaux were given in our previous research~\cite{FNS23}.

The rest of this paper is organized as follows. In Section~\ref{section2}, we introduce the shifted set-valued tableaux and the $K$-theoretic Schur $P$- and $Q$-functions. In Section~\ref{section3}, we show the main result. In Section~\ref{section4}, we extend the evaluations in Section~\ref{section3} to another skew version.
A summary and discussion of future work are presented in Section~\ref{section5}.

\section{Preliminaries}\label{section2}

In this section, we give the precise definitions of the shifted set-valued tableaux and the non-factorial $K$-theoretic Schur $P$- and $Q$-functions.
 Throughout this paper, we deal with the non-factorial version, i.e., the case in which all the extra parameters other than the symmetric variables are set to zero.

\subsection{Strict partitions and shifted Young diagrams}

A \emph{strict partition} is a sequence of positive integers
\begin{align*}
\lambda = (\lambda_1, \lambda_2, \dots, \lambda_r)
\end{align*}
such that $\lambda_1 > \lambda_2 > \cdots > \lambda_r > 0$.
Let $\mathcal{SP}$ denote the set of all strict partitions of positive integers, and let $\mathcal{SP}_l$ represent the set of strict partitions of a positive integer $l$. 
The \emph{length} $\ell(\lambda)$ of $\lambda$ is the largest positive integer $r$ such that $\lambda_r > 0$.
A strict partition in $\mathcal{SP}_l$ consists of $\ell(\lambda)$ distinct positive integers $\lambda_i$ such that their sum satisfies:
\begin{align*}
|\lambda| \coloneqq \sum_{i=1}^{r} \lambda_i = l.
\end{align*}
We can naturally identify a strict partition with a \emph{shifted Young diagram},
in which the $i$-th row is shifted to the right by $i-1$ positions and consists of $\lambda_i$ boxes.
Given a strict partition $\lambda$, the \textit{shifted Young diagram of shape} $\lambda$ is defined as follows:
\begin{align*}
{\rm SYD}_\lambda \coloneqq \{ (i,j) \in (\mathbb{Z}_{> 0})^2 \mid 1\leq i \leq \ell(\lambda),\ i \leq j \leq \lambda_i+i-1 \}.
\end{align*}
We denote the shifted Young diagram of shape $\lambda$ as the arrangement of $l$ boxes into $\ell(\lambda)$ rows, where the $i$-th row contains $\lambda_i$ boxes.
We describe the boxes of a shifted Young diagram using the matrix representation in the same manner, i.e., the first coordinate $i$ increases downward, and the second coordinate $j$ increases from left to right.
For example, the shifted Young diagram corresponding to the strict partition $\lambda = (4, 2, 1) \in \mathcal{SP}_7$ is obtained below:\vspace{3mm}
\begin{align*}
{\raisebox{-5.5pt}[0pt][0pt]{$\lambda =\ $}} {\raisebox{-0pt}[0pt][0pt]{\ytableaushort{\ \ \ \ ,\none\ \ ,\none\none\ }}}.\\[5mm]
\end{align*}

\subsection{Shifted semistandard tableaux and Schur $P$- and $Q$-functions}
We put $[n] \coloneqq \{ 1,2,\dots,n \}$.
For a positive integer $k \in [n]$, let $k' \coloneqq k - \frac{1}{2}$.
We denote $[n'] \coloneqq \{ 1', 2', \dots, n' \}$ and $[n', n] \coloneqq \{ 1', 1, 2', 2, \cdots, n', n \}$.

\begin{dfn}\label{sht}
A \textit{shifted semistandard tableau of shape $\lambda$} is a map $T : {\rm SYD}_\lambda \rightarrow 2^{[n',n]}$ satisfying the following conditions:
\begin{align*}
(1)&\ |T_{i,j}|= 1,\quad T_{i,j} \leq T_{i,j+1},\quad T_{i,j} \leq T_{i+1,j};\\
(2)&\ k \in [n]\ \text{appears at most once in each column};\\
(3)&\ k' \in [n']\ \text{appears at most once in each row};\\
(4)&\ T_{i,i} \subseteq [n].
\end{align*}
\end{dfn}
We call shifted semistandard tableaux simply shifted tableaux.
We write the set of all shifted tableaux of shape $\lambda$ as ${\rm ShT}_P(\lambda)$.
To indicate the number of variables $n$, we use the notation ${\rm ShT}_P(\lambda,n)$. 
Furthermore, the set of all shifted tableaux of shape $\lambda$ in $n$ variables, excluding the fourth condition (4), is denoted by ${\rm ShT}_Q(\lambda,n)$.
We only consider situations where ${\rm ShT}_P(\lambda,n)\neq \emptyset$ and ${\rm ShT}_Q(\lambda,n)\neq \emptyset$.
\begin{case}\label{exampleSVT}
Let $\lambda=(4,2,1) \in \mathcal{SP}_7,\ n=3$.
We take two shifted tableaux from ${\rm ShT}_P(\lambda,3)$ and ${\rm ShT}_Q(\lambda,3)$, respectively, as shown below:\vspace{3mm}
\begin{align*}
{\raisebox{-5.5pt}[0pt][0pt]{${\rm ShT}_P(\lambda,3)\ni\ $}} {\raisebox{-5.5pt}[0pt][0pt]{$T_1=$}} {\raisebox{1pt}[0pt][0pt]{\ytableaushort{1112,\none{2}2,\none\none3}},\ \ } {\raisebox{-5.5pt}[0pt][0pt]{$T_2=$}}
{\raisebox{1pt}[0pt][0pt]{\ytableaushort{{1}1{\ \!3'}{3},\none{2}{\ \!3'},\none\none{3}}}.}
\end{align*}
\vspace{20mm}
\begin{align*}
{\raisebox{-5.5pt}[0pt][0pt]{${\rm ShT}_Q(\lambda,3)\ni\ $}} {\raisebox{-5.5pt}[0pt][0pt]{$T_3=$}} {\raisebox{1pt}[0pt][0pt]{\ytableaushort{{\ \!1'}11{\ \!2'},\none{2}2,\none\none3}},\ \ }
{\raisebox{-5.5pt}[0pt][0pt]{$T_4=$}} {\raisebox{1pt}[0pt][0pt]{\ytableaushort{{\ \!1'}{\ \!2'}{\ \!3'}3,\none{\ \!2'}{\ \!3'},\none\none{3}}}.}\\[8mm]
\end{align*}
On the other hand, the two diagrams below are not elements of ${\rm ShT}_P(\lambda,3)$ or ${\rm ShT}_Q(\lambda,3)$.\vspace{5mm}
\begin{align*}
{\raisebox{1pt}[0pt][0pt]{\ytableaushort{{\ \!1'}1{\ \!2'}2,\none{2}{3},\none\none3}},\ \ }
{\raisebox{1pt}[0pt][0pt]{\ytableaushort{{1}1{\ \!2'}{3},\none{\ \!2'}{\ \!2'},\none\none{3}}}.}\\[8mm]
\end{align*}
\begin{align*}
{{\raisebox{20pt}[0pt][0pt]{(\scalebox{0.6}{{\raisebox{6pt}[0pt][0pt]{\ytableaushort{k,k,}}}}\ \text{is not allowed.})}}}\quad{{\raisebox{20pt}[0pt][0pt]{(\scalebox{0.6}{{\raisebox{-3pt}[0pt][0pt]{\ytableaushort{{k'\!}{k'\!}}}}}\ \text{is not allowed.})}}}
\end{align*}
\end{case}
We define the weight of a shifted tableau $T$ as
\begin{align}\label{weight}
\omega(T) = (\omega_1(T),\omega_2(T),\dots,\omega_n(T)) \in (\mathbb{Z}_{\geq 0})^n,
\end{align}
where $\omega_k(T) \coloneqq |\{ k',k \in [n',n] \mid k',k \in T \}|$.
The corresponding monomial with $x = (x_1, x_2, \dots, x_n) \in \mathbb{C}^n$ is defined by the weight $\omega(T)$ as follows:
\begin{align}\label{mono}
x^{\omega(T)} = x_1^{\omega_1(T)} x_2^{\omega_2(T)} \cdots x_n^{\omega_n(T)}.
\end{align}
The monomials $x^{\omega(T)}$ for $T=T_i$ ($i=1,2,3,4$) are given by
\begin{align*}
x^{\omega(T_1)}=x_1^3x_2^3x_3,\ x^{\omega(T_2)}=x_1^2x_2x_3^4,\ x^{\omega(T_3)}=x_1^3x_2^3x_3,\ x^{\omega(T_4)}=x_1x_2^2x_3^4.
\end{align*}
The Schur $P$- and $Q$-functions are defined as follows:
\begin{align}
P_\lambda(x) \coloneqq \sum_{T \in {\rm ShT}_P(\lambda,n)}x^{\omega(T)},\label{P}\\
Q_\lambda(x) \coloneqq \sum_{T \in {\rm ShT}_Q(\lambda,n)}x^{\omega(T)}.\label{Q}
\end{align}
The functions $P_\lambda(x)$ and $Q_{\lambda}(x)$ satisfy the following:
\begin{align}\label{P=Q}
Q_{\lambda}(x) = 2^{\ell(\lambda)}P_{\lambda}(x).
\end{align}
This equation can be checked from the definitions~\eqref{P} and~\eqref{Q}. 
Indeed, for a diagonal box containing an entry $k \in [n]$, replacing $k$ by $k' \in [n']$ does not change the monomial $x^{\omega(T)}$, 
since the weight $\omega_k(T)$ counts occurrences of $k$ and $k'$ together. 
There are $\ell(\lambda)$ diagonal boxes and two independent choices (primed or unprimed) for each entry. 
Each tableau contributing to $P_{\lambda}(x)$ corresponds to $2^{\ell(\lambda)}$ tableaux contributing to $Q_{\lambda}(x)$, which gives this relation.

\subsection{Shifted set-valued tableaux and $K$-theoretic Schur $P$- and   $Q$-functions}\label{dfn}
In this subsection, we first provide the definitions and examples of shifted semistandard set-valued tableaux based on~\cite{IN13}. Following this, we define the $K$-theoretic Schur $P$- and $Q$-functions as a sum over these tableaux.
\begin{dfn}[\cite{IN13}]\label{ssvt}
A \textit{shifted semistandard set-valued tableau of shape $\lambda$} is a map $T : {\rm SYD}_\lambda \rightarrow 2^{[n',n]}$ satisfying the following conditions:
\begin{align*}
(1)&\  |T_{i,j}| \geq 1,\quad \max{T_{i,j}} \leq \min{T_{i,j+1}},\quad \max{T_{i,j}} \leq \min{T_{i+1,j}};\\
(2)&\ k \in [n]\ \text{appears at most once in each column};\\
(3)&\ k' \in [n']\ \text{appears at most once in each row};\\
(4)&\ T_{i,i} \subseteq [n].
\end{align*}
\end{dfn}
We abbreviate shifted set-valued semistandard tableaux as shifted set-valued tableaux and denote the set of all shifted set-valued tableaux of shape $\lambda$ by ${\rm SSVT}_P(\lambda)$.
To specify the number of variables $n$, we use the notation ${\rm SSVT}_P(\lambda,n)$.
Additionally, the set of all shifted tableaux of shape $\lambda$ in $n$ variables, excluding the fourth condition (4) in Definition~\ref{ssvt}, is represented as ${\rm SSVT}_Q(\lambda,n)$.
We also only treat cases where ${\rm SSVT}_P(\lambda,n) \neq \emptyset$ and ${\rm SSVT}_Q(\lambda,n) \neq \emptyset$.
\begin{case}\label{exampleSSVT}
Let $\lambda=(4,2,1) \in \mathcal{SP}_7,\ n=3$.
We take two shifted set-valued tableaux from each of ${\rm SSVT}_P(\lambda,3)$ and ${\rm SSVT}_Q(\lambda,3)$ as follows:\vspace{5mm}
\begin{align*}
{\raisebox{-5.5pt}[0pt][0pt]{${\rm SSVT}_P(\lambda,3)\ni\ $}} {\raisebox{-5.5pt}[0pt][0pt]{$T^{(1)}=$}} {\raisebox{1pt}[0pt][0pt]{\ytableaushort{{1}1{1}{\ \!3'},\none{2}{23'\!},\none\none{3}}},\ \ } {\raisebox{-5.5pt}[0pt][0pt]{$T^{(2)}=$}}
{\raisebox{1pt}[0pt][0pt]{\ytableaushort{{1}{\ \!2'}{2}{23'\!},\none{2}{\ \!3'},\none\none{3}}.}}
\end{align*}
\vspace{8mm}
\begin{align*}
{\raisebox{-5.5pt}[0pt][0pt]{${\rm SSVT}_Q(\lambda,3)\ni\ $}} {\raisebox{-5.5pt}[0pt][0pt]{$T^{(3)}=$}} {\raisebox{1pt}[0pt][0pt]{\ytableaushort{{\ \!1'}{1}{1}{1},\none{\ \!2'}{2},\none\none{3'3}}},\ \ } {\raisebox{-5.5pt}[0pt][0pt]{$T^{(4)}=$}}{\raisebox{1pt}[0pt][0pt]{\ytableaushort{{\ \!1'}11{123},\none{2}{2},\none\none{\ \!3'}}},}\\[8mm]
\end{align*}
where for example, in $T^{(2)}$, $\scalebox{0.7}{{\raisebox{-1pt}[0pt][0pt]{\ytableaushort{{23'\!}}}}}$ indicates that the subset $T^{(2)}_{1,4} = \{2, 3'\} \subset [3', 3]$ is assigned to the box $(1,4)$ of $\lambda$.
Throughout this paper, we use only single-digit variables for examples of shifted set-valued (skew) tableaux.
\end{case}
We define the weight each of $T \in {\rm SSVT}_P(\lambda, n)$ and $T \in {\rm SSVT}_Q(\lambda, n)$ according to~\eqref{weight}.
The corresponding monomial $x^{\omega(T)}$ is defined as in~\eqref{mono}.
The monomials $x^{\omega(T)}$ for $T=T^{(i)}$ ($i=1,2,3,4$) are given by
\begin{align*}
x^{\omega(T^{(1)})}=x_1^3x_2^2x_3^3,\ x^{\omega(T^{(2)})}=x_1x_2^4x_3^3,\ x^{\omega(T^{(3)})}=x_1^4x_2^2x_3^2,\ x^{\omega(T^{(4)})}=x_1^4x_2^3x_3^2.
\end{align*}
Following~\cite{IN13}, {\it K-theoretic Schur P- and Q-functions} are defined as follows:
\begin{align}
GP_\lambda(x \mid \beta)&\coloneqq \sum_{T \in {\rm SSVT}_P(\lambda,n)}\beta^{|T|-|\lambda|}x^{\omega(T)},\label{GP}\\
GQ_\lambda(x \mid \beta)&\coloneqq \sum_{T \in {\rm SSVT}_Q(\lambda,n)}\beta^{|T|-|\lambda|}x^{\omega(T)}\label{GQ},
\end{align}
where $|T|$ represents the number of positive integers and positive half-integers assigned in $T$.
Setting $\beta = 0$ in~\eqref{GP} and~\eqref{GQ} yields $GP_{\lambda}(x \mid 0)$ and $GQ_{\lambda}(x \mid 0)$, which are $P_{\lambda}(x)$ and $Q_{\lambda}(x)$, respectively.

\subsection{Shifted set-valued skew tableaux and $K$-theoretic skew Schur $P$- and   $Q$-functions}

For two strict partitions $\lambda \in \mathcal{SP}_l$ and $\mu \in \mathcal{SP}_m$, with $l\geq m$ and $\lambda_i \geq \mu_i$, we define $\lambda / \mu \coloneqq (\lambda_1 - \mu_1, \lambda_2 - \mu_2, \dots, \lambda_{\ell(\mu)}-\mu_{\ell(\mu)},\lambda_{\ell(\mu)+1}, \dots, \lambda_{\ell(\lambda)} )$.
The set
\begin{align*}
{\rm SSYD}_{\lambda/\mu} \coloneqq \{ (i,j) \in (\mathbb{Z}_{\geq 0})^2 \mid 1\leq i \leq \ell(\lambda),\ \mu_i+i \leq j \leq \lambda_i + i - 1 \}.
\end{align*}
is called {\it shifted skew Young diagram of shape} $\lambda / \mu$.
We consider a given $\lambda / \mu$ and its associated shifted skew Young diagram to be equivalent.
The representation of $\lambda / \mu$ using the boxes $(i,j)$ is the same as that for shifted Young diagrams.
For example, for $\lambda = (6, 4, 3, 1) \in \mathcal{SP}_{14}$ and $\mu = (4, 2) \in \mathcal{SP}_6$, the shifted skew Young diagram $\lambda / \mu$ is obtained as shown below:\vspace{5mm}
\begin{equation}\label{skewexample}
{\raisebox{-11pt}[0pt][0pt]{$\lambda/\mu=\ $}}{\raisebox{0pt}[0pt][0pt]{\ytableaushort{\ \ \ \ \ \ ,\none\ \ \ \ ,\none\none\ \ \ ,\none\none\none\ }}}{\raisebox{-11pt}[0pt][0pt]{$\ -\ $}}{\raisebox{1pt}[0pt][0pt]{\ytableaushort{\ \ \ \ ,\none\ \ }}}{\raisebox{-11pt}[0pt][0pt]{$\ =\ $}}{\raisebox{0pt}[0pt][0pt]{\ytableaushort{\none\none\ \ ,\none\ \ ,\ \ \ ,\none\ }}.}\vspace{20mm}
\end{equation}
The diagonal positions of $\lambda/\mu$ are taken to be the same as those of $\lambda$.
For $\lambda/\mu$, define the {\it shifted semistandard skew tableau} as the diagram obtained by satisfying the conditions (1),(2),(3),(4) in Definition~\ref{sht}.
We denote the set of all shifted semistandard skew tableaux of shape $\lambda/\mu$ with $n$ variables as ${\rm ShT}_P(\lambda/\mu,n)$.
Let ${\rm ShT}_Q(\lambda/\mu,n)$ be the set of all shifted semistandard skew tableaux where the fourth condition (4) in Definition~\ref{sht} is excluded.
The {\it skew Schur} $P$- and {\it Schur} $Q$-{\it functions} are defined as follows:
\begin{align}
P_{\lambda/\mu}(x) \coloneqq \sum_{T \in {\rm ShT}_P(\lambda/\mu,n)}x^{\omega(T)},\label{P/}\\
Q_{\lambda/\mu}(x) \coloneqq \sum_{T \in {\rm ShT}_Q(\lambda/\mu,n)}x^{\omega(T)}.\label{Q/}
\end{align}

One type of generalization of $P_{\lambda/\mu}(x)$ and $Q_{\lambda/\mu}(x)$ was introduced in~\cite{LM21}. $P_{\lambda/\mu}(x)$ and $Q_{\lambda/\mu}(x)$ are specializations of skew Hall-Littlewood/Macdonald polynomials.
Given a $\lambda/\mu$, the {\it shifted semistandard set-valued skew tableau} is defined as the diagram that meets Definition~\ref{ssvt}.
We denote the set of all shifted semistandard set-valued skew tableaux of shape $\lambda / \mu$ having $n$ variables by ${\rm SSVT}_P(\lambda / \mu, n)$.
We assume that ${\rm SSVT}_P(\lambda / \mu, n) \neq \emptyset$.
Define the weight for $T \in {\rm SSVT}_P(\lambda/\mu,n)$ as~\eqref{weight} and the monomial $x^{\omega(T)}$ by~\eqref{mono}.
The set of all shifted tableaux of shape $\lambda/\mu$ in $n$ variables, excluding the fourth condition (4) from Definition~\ref{ssvt}, is denoted as ${\rm SSVT}_Q(\lambda/\mu,n)$.
For $T \in {\rm SSVT}_Q(\lambda/\mu,n)$, define $x^{\omega^{(T)}}$ in the same way.
Following~\cite{LM21}, the {\it K-theoretic skew Schur P- and Q-functions} are defined as follows:
\begin{align}
GP_{\lambda/\mu}(x \mid \beta)&\coloneqq \sum_{T \in {\rm SSVT}_P(\lambda/\mu,n)}\beta^{|T|-|\lambda/\mu|}x^{\omega(T)},\label{skewGP}\\
GQ_{\lambda/\mu}(x \mid \beta)&\coloneqq \sum_{T \in {\rm SSVT}_Q(\lambda/\mu,n)}\beta^{|T|-|\lambda/\mu|}x^{\omega(T)}\label{skewGQ}.
\end{align}

\section{Special values of $GP_{\lambda/\mu}$ and $GQ_{\lambda/\mu}$}\label{section3}
In this section, we show our main result.
In~\cite{FNS23}, a special value of the skew Grothendieck polynomial $G_{\lambda/\mu}$ is obtained using an involution on the set of set-valued tableaux ${\rm SVT}(\lambda/\mu,n)$:
\begin{thm}[\cite{FNS23}]
Let $\lambda$ and $\mu$ be partitions such that $\lambda \supset \mu$. 
We have
\begin{align*}
G_{\lambda/\mu}(\beta,\dots,\beta \mid -\beta^{-1}) = \beta^{|\lambda/\mu|}.
\end{align*}
\end{thm}
We prove the same result for $GP_{\lambda/\mu}$ by applying a similar argument and introducing an involution on the set ${\rm SSVT}_P(\lambda/\mu,n)$.
The proof for $GQ_{\lambda/\mu}$ follows analogously.
To derive the main result, we establish three lemmas.

\begin{lem}\label{lem1}
Let $\lambda$ and $\mu$ be strict partitions such that $\lambda \supset \mu$.
There exist unique tableaux $T_P \in {\rm SSVT}_P(\lambda/\mu,n)$ and $T_Q \in {\rm SSVT}_Q(\lambda/\mu,n)$ that satisfy the following conditions:
 \begin{align}
\sum_{(i,j) \in \lambda/\mu}\sum_{t \in (T_P)_{i,j}}t
  &= \
  \min_{T \in {\rm SSVT}_P(\lambda/\mu,n)}
  \sum_{(i,j) \in \lambda/\mu}\sum_{t \in T_{i,j}}t, \label{dfnT_P} \\
\sum_{(i,j) \in \lambda/\mu}\sum_{t \in (T_Q)_{i,j}}t
  &= \
  \min_{T \in {\rm SSVT}_Q(\lambda/\mu,n)}
  \sum_{(i,j) \in \lambda/\mu}\sum_{t \in T_{i,j}}t. \label{dfnT_Q}
\end{align}
Moreover, we have $T_P \in {\rm ShT}_P(\lambda/\mu,n)$ and $T_Q \in {\rm ShT}_Q(\lambda/\mu,n)$.
\end{lem}

\begin{proof}
We prove the uniqueness of $T_P$. 
The case of $T_Q$ can be shown in a similar manner.
For two distinct boxes $(i,j)$ and $(\tilde{i},\tilde{j})$ in $\lambda/\mu$, let the order be specified below:
\begin{align}\label{order}
&(i,j) < (\tilde{i},\tilde{j}) \quad\overset{\text{def}}{\Longleftrightarrow}\quad
  \begin{cases}
    j<\tilde{j},\\
    j=\tilde{j}\ \text{and}\ i<\tilde{i}.\\
  \end{cases}
\end{align}
Namely, boxes located further to the upper left of $\lambda/\mu$ are considered smaller in terms of the order.
We take two tableaux, $T_1, T_2 \in {\rm SSVT}_P(\lambda/\mu,n)$, that both minimize the sum of entries:
\begin{align*}
\min_{T \in {\rm SSVT}_P(\lambda/\mu,n)}
\sum_{(i,j) \in \lambda/\mu} \sum_{t \in T_{i,j}} t
=
\sum_{(i,j) \in \lambda/\mu} \sum_{t \in (T_1)_{i,j}} t
=
\sum_{(i,j) \in \lambda/\mu} \sum_{t \in (T_2)_{i,j}} t.
\end{align*}
We assume $T_1 \neq T_2$, then $|\{ (i,j) \in \lambda/\mu \mid (T_1)_{i,j} \neq (T_2)_{i,j} \}|\geq 1$.
Let $(i_0, j_0)$ be the minimal box of this set with respect to the order \eqref{order}.
We put $t_0 = \min \big( (T_1)_{i_0,j_0} \cup (T_2)_{i_0,j_0} \big)$.
By this definition, at least one of $\{ t_0 \} \neq (T_1)_{i_0,j_0}$ or $\{ t_0 \} \neq (T_2)_{i_0,j_0}$ must hold.
Without loss of generality, we assume $\{ t_0 \} \neq (T_1)_{i_0,j_0}$ and define a new tableau $T_0$ as follows:
\begin{align*}
(T_0)_{i,j} =
\begin{cases}
(T_1)_{i,j} & ((i,j) \neq (i_0,j_0)), \\
\{ t_0 \} & ((i,j) = (i_0,j_0)).
\end{cases}
\end{align*}
We verify that $T_0 \in {\rm SSVT}_P(\lambda/\mu,n)$.
Since $(i_0, j_0)$ is the minimal box where the entries of $T_1$ and $T_2$ differ, it follows that $(T_1)_{i,j} = (T_2)_{i,j}$ for any box $(i,j) < (i_0, j_0)$.
By the definition of $t_0$, we have either $t_0 = \min((T_1)_{i_0,j_0})$ or $t_0 = \min((T_2)_{i_0,j_0})$.
We first show that $T_0$ satisfies the condition (1) of Definition~\ref{ssvt} in the following.
\begin{itemize}
    \item If $(i_0,j_0-1) \in \lambda/\mu$, we have $\max((T_0)_{i_0,j_0-1}) = \max((T_1)_{i_0,j_0-1}) \le \min((T_1)_{i_0,j_0})$ since $T_1 \in {\rm SSVT}_P(\lambda/\mu,n)$.
    Therefore, it holds that $\max((T_0)_{i_0,j_0-1}) \le t_0 = \min((T_0)_{i_0,j_0})$.
     Similarly, if $(i_0-1,j_0) \in \lambda/\mu$, we can show that $\max((T_0)_{i_0-1,j_0}) \leq \min((T_0)_{i_0,j_0})$.
    \item If $(i_0,j_0+1) \in \lambda/\mu$,
    since $t_0 \le \max((T_1)_{i_0,j_0})$ and $\max((T_1)_{i_0,j_0}) \le \min((T_1)_{i_0,j_0+1})$, we have $\max((T_0)_{i_0,j_0}) = t_0 \le \min((T_1)_{i_0,j_0+1}) = \min((T_0)_{i_0,j_0+1})$.
    Similarly, if $(i_0+1,j_0) \in \lambda/\mu$, the condition $\max((T_0)_{i_0,j_0}) \leq \min((T_0)_{i_0+1,j_0})$ holds.
\end{itemize}
In addition, the condition (1) of Definition~\ref{ssvt} for $T_0$, except for the above cases, is satisfied due to the definition of $T_0$, and the fact that $T_1 \in {\rm SSVT}_P(\lambda/\mu,n)$. 

We confirm that $T_0$ satisfies conditions (2), (3), and (4) of Definition~\ref{ssvt}.
\begin{itemize}
    \item If $(i_0-1,j_0) \in \lambda/\mu$ and $k \in (T_0)_{i_0-1,j_0}\ (k \in [n])$, then we have $k \in (T_1)_{i_0-1,j_0}$ and $k \in (T_2)_{i_0-1,j_0}$ since $(i_0,j_0)$ is the minimal box in the order~\eqref{order}.
Therefore, we obtain $\min((T_1)_{i_0,j_0}) > k$ and $\min((T_2)_{i_0,j_0}) > k$ due to the semistandardness of $T_1$ and $T_2$. 
It follows that $t_0 \ne k$. 
In particular, we have $k \notin (T_0)_{i_0,j_0}$.

If $(i_0+1,j_0) \in \lambda/\mu$ and $k \in (T_0)_{i_0+1,j_0}$, then we have $k \in (T_1)_{i_0+1,j_0}$.
Due to the semistandardness of $T_1$, we find $\max((T_1)_{i_0,j_0}) < k$. 
This implies $t_0 \ne k$, consequently, $k \notin \{ t_0 \} =(T_0)_{i_0,j_0}$.

In conclusion, $T_0$ satisfies the condition (2) in Definition~\ref{ssvt}.
The condition (3) in Definition~\ref{ssvt} for $T_0$ can be checked by the same way.
    \item If $(i_0,j_0)$ is a diagonal box, it follows $t_0 \in [n]$ by $T_1, T_2 \in {\rm SSVT}_P(\lambda/\mu,n)$ and the definition of $t_0$. 
Hence, $T_0$ satisfies the condition (4) of Definition~\ref{ssvt}.
\end{itemize}
Thus, we conclude that $T_0 \in {\rm SSVT}_P(\lambda/\mu,n)$.

From our assumption $\{ t_0 \} \neq (T_1)_{i_0,j_0}$, the set $(T_1)_{i_0,j_0}$ either contains an element larger than $t_0$ or consists of multiple elements.
Hence, we have
\begin{align*}
\left(\sum_{t \in (T_0)_{i_0,j_0}}t\right) \ =\  t_0 <  \left(\sum_{t \in (T_1)_{i_0,j_0}}t\right).
\end{align*}
This leads to the inequality:
\begin{align*}
\sum_{(i,j) \in \lambda/\mu} \sum_{t \in (T_0)_{i,j}} t < \sum_{(i,j) \in \lambda/\mu} \sum_{t \in (T_1)_{i,j}} t.
\end{align*}
This contradicts the assumption that $T_1$ minimizes the sum of entries.
Therefore, we obtain $T_1 = T_2$, which proves the uniqueness of $T_P$.

We show that $|(T_P)_{i,j}|=1$ for all $(i,j)$.
Assume there exists a box $(i_1,j_1)$ such that $|(T_P)_{i_1,j_1}| > 1$.
We consider a new tableau $T'_P$ where the entry in $(i_1,j_1)$ is replaced by the set $\{ \min((T_P)_{i_1,j_1}) \}$, and all other entries remain the same as in $T_P$.
This new tableau $T'_P$ still satisfies the conditions for ${\rm SSVT}_P(\lambda/\mu,n)$, but the sum of its entries is strictly smaller than that of $T_P$.
This contradicts the minimality of $T_P$.
Thus, we must have $|(T_P)_{i,j}|=1$ for all boxes $(i,j)$, which implies $T_P \in {\rm ShT}_P(\lambda/\mu,n)$.

The proof for $T_Q$ is identical, with the exception that condition (4) for diagonal boxes need not be considered.
This completes the proof.
\end{proof}

\begin{rem}
By Lemma~\ref{lem1}, since $T_P \in {\rm ShT}_P(\lambda/\mu,n)$, we sometimes identify $(T_P)_{i,j}$ with its element.
\end{rem}

\begin{lem}\label{lem2}
Let $\lambda$ and $\mu$ be strict partitions such that $\lambda \supset \mu$.
For any $T \in  {\rm SSVT}_P(\lambda/\mu,n)$ (resp. $T \in  {\rm SSVT}_Q(\lambda/\mu,n)$) and $(i,j) \in \lambda/\mu$, we have
\begin{align}
(T_P)_{i,j} \leq \min(T_{i,j})\label{minTP},\\
(T_Q)_{i,j} \leq \min(T_{i,j})\label{minTQ}.
\end{align}
\end{lem}
\begin{proof}
We take $T \in  {\rm SSVT}_P(\lambda/\mu,n)$ and assume that there exists a box that does not satisfy \eqref{minTP}.
Let $(\tilde{i},\tilde{j}) \in \lambda/\mu$ be the minimal such box under the order \eqref{order}.
We define a tableau $\widetilde{T}$ as follows:
\begin{align*}
\widetilde{T}_{i,j} =
\begin{cases}
(T_P)_{i,j} & ((i,j) \neq (\tilde{i},\tilde{j})), \\
\{\min(T_{\tilde{i},\tilde{j}})\} & ((i,j) = (\tilde{i},\tilde{j})).
\end{cases}
\end{align*}
We show that $\widetilde{T} \in {\rm SSVT}_P(\lambda/\mu,n)$.
If $(\tilde{i},\tilde{j}-1) \in \lambda/\mu$, it follows from the minimality of $(\tilde{i},\tilde{j})$ that
\begin{align*}
\max(\widetilde{T}_{\tilde{i},\tilde{j}-1}) = (T_P)_{\tilde{i},\tilde{j}-1} &\leq \min(T_{\tilde{i},\tilde{j}-1})\\
&\leq  \max(T_{\tilde{i},\tilde{j}-1}) \leq \min (T_{\tilde{i},\tilde{j}}) = \min (\widetilde{T}_{\tilde{i},\tilde{j}}).
\end{align*}
Thus, we have $\max(\widetilde{T}_{\tilde{i},\tilde{j}-1}) \leq \min (\widetilde{T}_{\tilde{i},\tilde{j}})$.
Similarly, if $(\tilde{i}-1,\tilde{j}) \in \lambda/\mu$, we can prove that $\max(\widetilde{T}_{\tilde{i}-1,\tilde{j}}) \leq \min (\widetilde{T}_{\tilde{i},\tilde{j}})$.

If $(\tilde{i},\tilde{j}+1) \in \lambda/\mu$, from the definition of $\widetilde{T}$ and the assumption on $(\tilde{i},\tilde{j})$, we obtain
\begin{align}
\max(\widetilde{T}_{\tilde{i},\tilde{j}})=\min(T_{\tilde{i},\tilde{j}}) < (T_P)_{\tilde{i},\tilde{j}} &\leq (T_P)_{\tilde{i},\tilde{j}+1} = \min (\widetilde{T}_{\tilde{i},\tilde{j}+1}).\label{cont}
\end{align}
This implies that $\max(\widetilde{T}_{\tilde{i},\tilde{j}}) < \min (\widetilde{T}_{\tilde{i},\tilde{j}+1})$.
Likewise, if $(\tilde{i}+1,\tilde{j}) \in \lambda/\mu$, it holds that $\min (\widetilde{T}_{\tilde{i},\tilde{j}})<\min(\widetilde{T}_{\tilde{i}+1,\tilde{j}})$.
Hence, $\widetilde{T}$ satisfies condition (1) of Definition~\ref{ssvt}.
The conditions (2), (3), and (4) of Definition~\ref{ssvt} can be checked analogously, following the same argument as in the proof of Lemma~\ref{lem1}.
Therefore, we find that $\widetilde{T} \in {\rm SSVT}_P(\lambda/\mu,n)$.
By the definition of $\widetilde{T}$ and the inequality from \eqref{cont}, we have 
\begin{align*}
\sum_{(i,j) \in \lambda/\mu}\sum_{t \in \widetilde{T}_{i,j}}t < \sum_{(i,j) \in \lambda/\mu}\sum_{t \in (T_P)_{i,j}}t.
\end{align*}
This contradicts the condition~\eqref{dfnT_P}  of Lemma~\ref{lem1} for $T_P$.
Thus, the equation~\eqref{minTP} must be satisfied.
The claim~\eqref{minTQ} is demonstrated in the same manner, without needing to consider condition (4) of Definition~\ref{ssvt}.
This establishes the proof.
\end{proof}

We set
\begin{align*}
{\rm SSVT}_P'(\lambda/\mu,n) \coloneqq {\rm SSVT}_P(\lambda/\mu,n)\setminus \{ T_P \},\\ {\rm SSVT}_Q'(\lambda/\mu,n) \coloneqq {\rm SSVT}_Q(\lambda/\mu,n) \setminus \{ T_Q \}.
\end{align*}
For $T \in {\rm SSVT}_P'(\lambda/\mu,n)$, let $(i,j)$ denote the smallest box in $\lambda/\mu$ with respect to the order~\eqref{order} such that $T_{i,j} \neq (T_P)_{i,j}$.
We define a map $\iota_P : {\rm SSVT}_P'(\lambda/\mu,n) \to {\rm SSVT}_P'(\lambda/\mu,n)$ as follows (see also Example~\ref{skewexamplesri} for an explicit example):
\begin{align}\label{iotaP}
  \begin{split}
    &\iota_P : {\rm SSVT}_P'(\lambda/\mu,n) \rightarrow {\rm SSVT}_P'(\lambda/\mu,n) \\
    &T_{i,j} \mapsto
  \begin{cases}
    T_{i,j} \setminus (T_P)_{i,j} & (T_{i,j}\cap (T_P)_{i,j} \neq \emptyset),\\
    T_{i,j} \sqcup (T_P)_{i,j} & (T_{i,j}\cap (T_P)_{i,j} = \emptyset),
  \end{cases}\\
   &T_{\tilde{i},\tilde{j}} \mapsto T_{\tilde{i},\tilde{j}}\ \  ((\tilde{i},\tilde{j})\neq(i,j)).
  \end{split}
\end{align}
We define a map $\iota_Q:{\rm SSVT}_Q'(\lambda/\mu,n)\rightarrow{\rm SSVT}_Q'(\lambda/\mu,n)$ analogously to $\iota_P$.
\begin{lem}\label{lem3}
The map $\iota_P$ (resp. $\iota_Q$) is well-defined and an involution on ${\rm SSVT}_P'(\lambda/\mu,n)$ (resp. ${\rm SSVT}_Q'(\lambda/\mu,n)$).
In particular, $\iota_P$ and $\iota_Q$ are bijections and $\{ \iota_P(T) \mid T \in {\rm SSVT}'_P(\lambda/\mu,n) \}=\{ T \mid T \in {\rm SSVT}'_P(\lambda/\mu,n) \}$, $\{ \iota_Q(T) \mid T \in {\rm SSVT}'_Q(\lambda/\mu,n) \}=\{ T \mid T \in {\rm SSVT}'_Q(\lambda/\mu,n) \}$.
\end{lem}

\begin{proof}

We first show that $(\iota_P(T))_{i,j}$ is non-empty. 
This is obvious for the case of $T_{i,j}\cap(T_P)_{i,j} = \emptyset$. 
We assume $T_{i,j} \cap (T_P)_{i,j} \neq \emptyset$.
If $|T_{i,j}|=1$, Lemma \ref{lem2} implies that $T_{i,j}=(T_P)_{i,j}$.
This contradicts the definition of $(i,j)$. 
Thus, we must have $|T_{i,j}|\geq 2$. 
Since $|(T_P)_{i,j}|=1$, it holds that $(\iota_P(T))_{i,j}=T_{i,j}\setminus(T_P)_{i,j}\neq\emptyset$. 
Therefore $(\iota_P(T))_{i,j}$ is non-empty.

Next, we prove the map $\iota_P$ is well-defined; that is, $\iota_P(T)\in{\rm SSVT}_P'(\lambda/\mu,n)$.
\begin{itemize}
\item If $(i-1,j)\in\lambda/\mu$, the minimality of the box $(i,j)$ ensures $(\iota_P(T))_{i-1,j}=T_{i-1,j}=(T_P)_{i-1,j}$. 
By Lemma~\ref{lem2}, we have 
\begin{align*}
(T_P)_{i,j} = \min(T_{i,j} \cup (T_P)_{i,j}),\quad(T_P)_{i,j} < \min(T_{i,j} \setminus (T_P)_{i,j}).
\end{align*}
Hence, it holds that $(T_P)_{i-1,j}\leq\min((\iota_P(T))_{i,j})$.
Thus, we obtain
\begin{align*}
\max((\iota_P(T))_{i-1,j}) = (T_P)_{i-1,j} \leq \min((\iota_P(T))_{i,j}).
\end{align*}
The condition $\max((\iota_P(T))_{i-1,j}) \leq \min((\iota_P(T))_{i,j})$ is therefore satisfied. 
A similar argument applies if $(i,j-1) \in \lambda/\mu$.
In this case, the condition $\max((\iota_P(T))_{i,j-1}) \leq \min((\iota_P(T))_{i,j})$ follows.

\item If $(i,j+1) \in \lambda/\mu$, we have $\max((\iota_P(T))_{i,j}) = \max(T_{i,j})$ by the definition of $\iota_P$ and Lemma~\ref{lem2}. 
Since $T \in {\rm SSVT}_P'(\lambda/\mu,n)$, we find $\max(T_{i,j}) \le \min(T_{i,j+1})$.  It follows from $\iota_P(T_{i,j+1})=T_{i,j+1}$ that
\begin{align*}
\max((\iota_P(T))_{i,j}) = \max(T_{i,j}) \leq \min(T_{i,j+1}) = \min((\iota_P(T))_{i,j+1}).
\end{align*}
Hence, we have $\max((\iota_P(T))_{i,j}) \leq \min((\iota_P(T))_{i,j+1})$. 
An analogous argument shows that if $(i+1,j) \in \lambda/\mu$, the condition $\max((\iota_P(T))_{i,j}) \leq \min((\iota_P(T))_{i+1,j})$ holds.
Therefore, we confirm that $\iota_P(T)$ satisfies condition (1) of Definition~\ref{ssvt}.
\end{itemize}

We verify that $\iota_P(T)$ satisfies conditions (2), (3), and (4) of Definition~\ref{ssvt}.
\begin{itemize}
\item If $(i-1,j)\in\lambda/\mu$ and $k\in(\iota_P(T))_{i-1,j}\ (k \in [n])$, then by construction of $\iota_P$, it holds that $k\in T_{i-1,j}$. 
From the definition of $(i,j)$, we have $T_{i-1,j}=(T_P)_{i-1,j}$.
This implies that $k=(T_P)_{i-1,j}$.
By the semistandardness of $T_P$, it follows that $k=(T_P)_{i-1,j} < (T_P)_{i,j}$. 
This shows
\begin{align*}
k < (T_P)_{i,j} \leq \min((\iota_P(T))_{i,j}).
\end{align*}
Thus, we find $k\notin(\iota_P(T))_{i,j}$. 
A similar argument applies if $(i,j-1)\in\lambda/\mu$. 

We assume $(i,j+1) \in \lambda/\mu$ and $k' \in (\iota_P(T))_{i,j}$ ($k' \in [n']$). 
We recall that $T, T_P \in {\rm SSVT}_P(\lambda/\mu,n)$, and the map $\iota_P$ modifies $T_{i,j}$ by either adding or removing $(T_P)_{i,j}$. 
Thus, it follows that $k' \in T_{i,j}$ or $k' \in (T_P)_{i,j}$.
If $k' \in T_{i,j}$, the semistandardness of $T$ immediately implies $k' \notin T_{i,j+1}$.
If $k' \in (T_P)_{i,j}$, from the definition of $\iota_P$ and the assumption $k' \in (\iota_P(T))_{i,j}$, it holds that $(\iota_P(T))_{i,j} = T_{i,j} \sqcup (T_P)_{i,j}$. 
Since $T_{i,j} \cap (T_P)_{i,j} = \emptyset$, it follows from Lemma~\ref{lem2} that $k' = (T_P)_{i,j} < \min(T_{i,j})$.
On the other hand, the semistandardness of $T$ implies $\max(T_{i,j}) \leq \min(T_{i,j+1})$. 
This leads to
    \begin{align*}
    k' < \min(T_{i,j}) \leq \max(T_{i,j}) \leq \min(T_{i,j+1}).
    \end{align*}
Hence, $k' \notin T_{i,j+1}$, and therefore $k' \notin (\iota_P(T))_{i,j+1}$ by the definition of $\iota_P$.
A similar argument shows that if $(i+1,j) \in \lambda/\mu$ and $k \in [n]$ with $k \in (\iota_P(T))_{i,j}$, then $k \notin (\iota_P(T))_{i+1,j}$.
Thus, $\iota_P(T)$ satisfies conditions (2) and (3) of Definition~\ref{ssvt}.

\item If $(i,j)$ is a diagonal box, since both $T$ and $T_P$ are in ${\rm SSVT}_P(\lambda/\mu,n)$, their entries in this box do not contain any element $k' \in [n']$. Consequently, the modified entry $(\iota_P(T))_{i,j}$ also does not contain $k'$. Thus, $\iota_P(T)$ satisfies condition (4) of Definition~\ref{ssvt}.
\end{itemize}
Furthermore, we have $(\iota_P(T))_{i,j}\neq(T_P)_{i,j}$ by the definition of $\iota_P$.
This inequality ensures that $\iota_P(T)\neq T_P$. 
We therefore conclude that $\iota_P(T)\in{\rm SSVT}_P'(\lambda/\mu,n)$.

We prove that $\iota_P$ is an involution on ${\rm SSVT}_P'(\lambda/\mu,n)$. Let $T'=\iota_P(T)$. 
By definition, the smallest box where $T$ and $T_P$ differ is $(i,j)$. 
For any box $(\hat{i},\hat{j})$ smaller than $(i,j)$, we have
\begin{align*}
(T_P)_{\hat{i}, \hat{j}} = T_{\hat{i}, \hat{j}} = T'_{\hat{i}, \hat{j}}.
\end{align*}
We already show that $T'_{i,j}\neq(T_P)_{i,j}$.
Hence, when applying $\iota_P$ to $T'$, the chosen box must be $(i,j)$. 
We examine the action of applying $\iota_P$ to $T'$.
\begin{itemize}
\item If $T_{i,j}\cap(T_P)_{i,j}\neq\emptyset$, it follows that $T'_{i,j}=T_{i,j}\setminus(T_P)_{i,j}$, which implies that $T'_{i,j}\cap(T_P)_{i,j}=\emptyset$.
Applying $\iota_P$ to $T'$ yields
\begin{align*}
(\iota_P(T'))_{i,j} = T'_{i,j} \sqcup (T_P)_{i,j} = (T_{i,j} \setminus (T_P)_{i,j}) \sqcup (T_P)_{i,j} = T_{i,j}.
\end{align*}
\item If $T_{i,j}\cap (T_P)_{i,j} = \emptyset$, we have $T'_{i,j}=T_{i,j}\sqcup(T_P)_{i,j}$, which indicates that $T'_{i,j}\cap(T_P)_{i,j} \neq \emptyset$.
Applying $\iota_P$ to $T'$ gives
\begin{align*}
(\iota_P(T'))_{i,j} = T'_{i,j} \setminus (T_P)_{i,j} = (T_{i,j} \sqcup (T_P)_{i,j}) \setminus (T_P)_{i,j} = T_{i,j}.
\end{align*}
\end{itemize}
For any other box $(\tilde{i},\tilde{j})\neq(i,j)$, it holds that $(\iota_P(\iota_P(T)))_{\tilde{i},\tilde{j}} = (\iota_P(T))_{\tilde{i},\tilde{j}} = T_{\tilde{i},\tilde{j}}$ by the definition of $\iota_P$.
Thus, since $\iota_P(\iota_P(T))=T$ for any $T\in{\rm SSVT}_P'(\lambda/\mu,n)$, it follows that $\iota_P$ is an involution on ${\rm SSVT}_P'(\lambda/\mu,n)$. 
In particular, $\iota_P$ is a bijection. 
Therefore, we have
\begin{align*}
\{ \iota_P(T) \mid T \in {\rm SSVT}'_P(\lambda/\mu,n)\} = \{ T \mid T \in {\rm SSVT}'_P(\lambda/\mu,n) \}.
\end{align*}

The proof that $\iota_Q$ is a well-defined involution on ${\rm SSVT}_Q'(\lambda/\mu,n)$ is analogous to that of $\iota_P$, with the exception of the argument for part (4) of Definition~\ref{ssvt}.
Since $\iota_Q$ is an involution, it is a bijection, and we obtain
\begin{align*}
\{ \iota_Q(T) \mid T \in {\rm SSVT}'_Q(\lambda/\mu,n) \}=\{ T \mid T \in {\rm SSVT}'_Q(\lambda/\mu,n) \}.
\end{align*}
This concludes the proof.
\end{proof}

\begin{thm}\label{skewspecialvalue}
Let $\lambda$ and $\mu$ be strict partitions such that $\lambda \supset \mu$.
We have
\begin{align}
GP_{\lambda/\mu}(\beta,\dots,\beta \mid -\beta^{-1})=\beta^{|\lambda/\mu|},\label{skewGPB}\\
GQ_{\lambda/\mu}(\beta,\dots,\beta \mid -\beta^{-1})=\beta^{|\lambda/\mu|}.\label{skewGQB}
\end{align}
\end{thm}
\begin{proof}
We prove the equation~(\ref{skewGPB}). 
By the definition of $GP_{\lambda/\mu}(x \mid \beta)$ in~\eqref{skewGP} and Lemma~\ref{lem1}, we have
\begin{align}
GP_{\lambda/\mu}(\beta,\dots,\beta\mid{-\beta^{-1}}) = \beta^{|T_P|} + \sum_{T \in {\rm SSVT}_P'(\lambda/\mu,n)}({-\beta^{-1}})^{|T|-|\lambda/\mu|}\beta^{|T|}.\label{main1}
\end{align}
By Lemma~\ref{lem3}, the map $\iota_P$ is an involution on ${\rm SSVT}'_P(\lambda/\mu,n)$ and also satisfies $|\iota_P(T)|=|T| \pm 1$ for $T \in {\rm SSVT}'_P(\lambda/\mu,n)$.
It follows that
\begin{align*}
\sum_{T \in {\rm SSVT}_P'(\lambda/\mu,n)}({-\beta^{-1}})^{|T|-|\lambda/\mu|}\beta^{|T|}
&=\sum_{T \in {\rm SSVT}_P'(\lambda/\mu,n)}({-\beta^{-1}})^{|\iota_P(T)|-|\lambda/\mu|}\beta^{|\iota_P(T)|}\\
&=\sum_{T \in {\rm SSVT}_P'(\lambda/\mu,n)}({-\beta^{-1}})^{|T|\pm1-|\lambda/\mu|}\beta^{|T|\pm1}\\
&=-\sum_{T \in {\rm SSVT}_P'(\lambda/\mu,n)}({-\beta^{-1}})^{|T|-|\lambda/\mu|}\beta^{|T|}.
\end{align*}
From this, we obtain
\begin{align*}
\sum_{T \in {\rm SSVT}_P'(\lambda,n)}({-\beta^{-1}})^{|T|-|\lambda/\mu|}\beta^{|T|}=0.
\end{align*}
Substituting this result into the equation~(\ref{main1}), we obtain
\begin{align*}
GP_{\lambda/\mu}(\beta,\dots,\beta\mid{-\beta^{-1}}) = \beta^{|T_P|} + 0 = \beta^{|\lambda/\mu|}.
\end{align*}
We remark that $|T_P|=|\lambda/\mu|$ since $T_P \in {\rm ShT}_P(\lambda/\mu,n)$.
Therefore, the equation~(\ref{skewGPB}) is valid.

By an analogous argument and Lemma~\ref{lem3}, we have that the map $\iota_Q$ is an involution on ${\rm SSVT}'_Q(\lambda/\mu,n)$ and satisfies $|\iota_Q(T)|=|T| \pm 1$ for $T \in {\rm SSVT}'_Q(\lambda/\mu,n)$. 
This implies that
\begin{align*}
\sum_{T \in {\rm SSVT}_Q'(\lambda/\mu,n)}({-\beta^{-1}})^{|T|-|\lambda/\mu|}\beta^{|T|}=0.
\end{align*}
Hence, the equation~(\ref{skewGQB}) follows. 
This completes the proof.
\end{proof}

\begin{case}\label{skewexamplesri}
Consider the shifted Young diagrams provided in \eqref{skewexample}. The tableaux $T_P \in {\rm SSVT}_P(\lambda/\mu, 5)$ and $T_Q \in {\rm SSVT}_Q(\lambda/\mu, 5)$ are given by:
\vspace{5mm}
\begin{align*}
{\raisebox{-11pt}[0pt][0pt]{$T_P=\ $}}{\raisebox{0pt}[0pt][0pt]{\ytableaushort{\none\none{\ \!1'}1,\none{\ \!1'}1,11{\ \!2'},\none2}},\quad}
{\raisebox{-11pt}[0pt][0pt]{$T_Q=\ $}}{\raisebox{0pt}[0pt][0pt]{\ytableaushort{\none\none{\ \!1'}1,\none{\ \!1'}1,{\ \!1'}1{\ \!2'},\none{\ \!2'}}}.\quad}
\end{align*}
\vspace{15mm}

\noindent
We take $T^{\langle1\rangle}$ and $T^{\langle2\rangle} \in {\rm SSVT}'_P(\lambda/\mu,5)$ below and apply $\iota_P$ to them. The results $\iota_P(T^{\langle1\rangle})$ and $\iota_P(T^{\langle2\rangle})$ are as follows:
\vspace{5mm}
\begin{align*}
{\raisebox{-11pt}[0pt][0pt]{$T^{\langle1\rangle}=\ $}}{\raisebox{0pt}[0pt][0pt]{\ytableaushort{\none\none{34'\!}{45'\!},\none{\ \!1'}{4'4},1{{\color{red}{1}}4}{\ \!5'},\none{5}}}\quad}
{\raisebox{-11pt}[0pt][0pt]{$\longrightarrow\quad$}}
{\raisebox{-11pt}[0pt][0pt]{$\iota_P(T^{\langle1\rangle})=\ $}}{\raisebox{0pt}[0pt][0pt]{\ytableaushort{\none\none{34'\!}{45'\!},\none{\ \!1'}{4'4},1{4}{\ \!5'},\none{5}}}.}\\[10mm]
\end{align*}
\vspace{5mm}
\begin{align*}
{\raisebox{-11pt}[0pt][0pt]{$T^{\langle2\rangle}=\ $}}{\raisebox{0pt}[0pt][0pt]{\ytableaushort{\none\none{\ \!1'}{3'5},\none{\ \!1'}{1},1{1}{{\color{blue}{4}}},\none{2}}}\quad}
{\raisebox{-11pt}[0pt][0pt]{$\longrightarrow\quad$}}
{\raisebox{-11pt}[0pt][0pt]{$\iota_P(T^{\langle2\rangle})=\ $}}{\raisebox{0pt}[0pt][0pt]{\ytableaushort{\none\none{\ \!1'}{3'5},\none{\ \!1'}{1},1{1}{{\color{red}{2'}}{\color{blue}{4}}},\none{2}}}.}\\[15mm]
\end{align*}
Similarly, applying the map $\iota_Q$ to the following $T^{\langle3\rangle}$ and $T^{\langle4\rangle} \in{\rm SSVT}'_Q(\lambda/\mu,5)$ yields:
\vspace{5mm}
\begin{align*}
{\raisebox{-11pt}[0pt][0pt]{$T^{\langle3\rangle}=\ $}}{\raisebox{0pt}[0pt][0pt]{\ytableaushort{\none\none{\ \!1'}{{\color{red}{1}}23},\none{\ \!1'}{1},{\ \!1'}{1}{\ \!2'},\none{\ \!2'}}}\quad}
{\raisebox{-11pt}[0pt][0pt]{$\longrightarrow\quad$}}
{\raisebox{-11pt}[0pt][0pt]{$\iota_Q(T^{\langle3\rangle})=\ $}}{\raisebox{0pt}[0pt][0pt]{\ytableaushort{\none\none{\ \!1'}{23},\none{\ \!1'}{1},{\ \!1'}{1}{\ \!2'},\none{\ \!2'}}}.}\\[10mm]
\end{align*}
\vspace{5mm}
\begin{align*}
{\raisebox{-11pt}[0pt][0pt]{$T^{\langle4\rangle}=\ $}}{\raisebox{0pt}[0pt][0pt]{\ytableaushort{\none\none{{\color{blue}{3'\!}}}{345},\none{\ \!1'}{3'\!4'\!},{\ \!1'}{1}{4'5},\none{\ \!2'}}}\quad}
{\raisebox{-11pt}[0pt][0pt]{$\longrightarrow\quad$}}
{\raisebox{-11pt}[0pt][0pt]{$\iota_Q(T^{\langle4\rangle})=\ $}}{\raisebox{0pt}[0pt][0pt]{\ytableaushort{\none\none{{\color{red}{1'}}{\color{blue}{3'}}}{345},\none{\ \!1'}{3'\!4'\!},{\ \!1'}{1}{4'5},\none{\ \!2'}}}.}\\[15mm]
\end{align*}
\end{case}

We obtain the following result as a corollary of Theorem~\ref{skewspecialvalue}.
\begin{cor}\label{oddness}
Let $\lambda$ and $\mu$ be strict partitions such that $\lambda \supset \mu$.
Both $|{\rm SSVT}_P(\lambda/\mu,n)|$ and $|{\rm SSVT}_Q(\lambda/\mu,n)|$ are odd.
\end{cor}
\begin{proof}
We prove the claim for $|{\rm SSVT}_P(\lambda/\mu,n)|$.
\begin{align*}
|{\rm SSVT}_P(\lambda/\mu,n)|&=GP_{\lambda/\mu}(1,1,\dots,1\mid 1)\\
&\equiv GP_{\lambda/\mu}(1,1,\dots,1\mid -1)\pmod 2 \\
&=1\quad ({\rm Theorem}~\ref{skewspecialvalue}).
\end{align*}
It follows that $|{\rm SSVT}_P(\lambda/\mu,n)|$ is odd.
A similar argument applies to $|{\rm SSVT}_Q(\lambda/\mu,n)|$.
\end{proof}
\begin{rem}\label{remsri}
Since the maps $\iota_P$ and $\iota_Q$ satisfy $\iota_P^2=\iota_Q^2={\rm id}$ and $(-1)^{|\iota_P(T)|}=(-1)^{|\iota_Q(T)|}=-(-1)^{|T|}$, these maps are called \textit{sign-reversing involutions}.
\end{rem}
\begin{rem}\label{reneyd}
In~\cite{IN13}, {\it excited Young diagrams} were introduced to provide another combinatorial representation of $K$-theoretic Schur $P$- and $Q$-functions as follows:
\begin{align*}
GP_{\lambda}(x \mid \beta) &= \sum_{D\in \mathcal{E}_n^{B}(\lambda)}wt^{B}(D),\\
GQ_{\lambda}(x \mid \beta) &= \sum_{D\in \mathcal{E}_n^{C}(\lambda)}wt^{C}(D),
\end{align*}
where $X$ is either $B$ or $C$, and
\begin{align*}
wt^{X}(D)&=\prod_{(i,j) \in D}wt^{X}(i,j)\prod_{(\tilde{i},\tilde{j}) \in \mathcal{B}^{X}(D)}\left(1+\beta wt^X(\tilde{i},\tilde{j})\right).
\end{align*}
We refer to~\cite{IN13} for the details of the symbols and definitions.
By setting $x_1=\cdots=x_n=\beta$ and applying $\beta \mapsto -\beta^{-1}$, we have $wt^{X}(D)=0$ unless an excited Young diagram $D \in \mathcal{E}_n^{X}(\lambda)$ is an ordinary Young diagram.
From this, we can confirm that the equations~\eqref{skewGPB} and~\eqref{skewGQB} hold in the case of a non-skew Young diagram.
\end{rem}

\section{Special values of $GP_{\lambda/\!\!/\mu}$ and $GQ_{\lambda/\!\!/\mu}$}\label{section4}

In this section, we discuss another type of skew version of the $K$-theoretic Schur $P$- and $Q$-functions.
We denote a symmetric function $f$ with symmetric variables $x, y \in \mathbb{C}^n$ as $f(x, y)$.
It is known that the following branching formulas hold \cite{Mac95}:
\begin{align}
P_{\lambda}(x,y)&= \sum_{\nu} P_{\nu}(x)P_{\lambda/\nu}(y),\label{P(x,y)}\\
Q_{\lambda}(x,y)&= \sum_{\nu} Q_{\nu}(x)Q_{\lambda/\nu}(y),\label{Q(x,y)}
\end{align}
where in both sums $\nu$ ranges over all strict partitions.
These identities~\eqref{P(x,y)} and~\eqref{Q(x,y)} follow from the definition of $P_{\lambda}(x)$ and $Q_{\lambda}(x)$, respectively.
To introduce the $K$-theoretic version of the identities~\eqref{P(x,y)} and~\eqref{Q(x,y)}, the following functions $GP_{\lambda/\!\!/\mu}(x \mid \beta)$ and $GQ_{\lambda/\!\!/\mu}(x \mid \beta)$ are defined in~\cite{CM23}.
For $\mu \in \mathcal{SP}_m$, we call the position $(i,j) \in \mu$ {\it removable box} if $\mu \setminus \{(i,j)\} \in \mathcal{SP}_{m-1}$.
Let ${\rm Rem}(\mu)$ denote all the removable boxes of $\mu$ and its cardinality by $a \coloneqq |{\rm Rem}(\mu)|$.
Following~\cite{CM23} (see also~\cite{LM21,M24}), for $\lambda \supseteq \mu$, define
\begin{align}
GP_{\lambda/\!\!/\mu}(x \mid \beta) &\coloneqq \sum_{\nu \subseteq \mu} \beta^{|\mu/\nu|}GP_{\lambda/\nu}(x \mid \beta),\label{GP//}\\
GQ_{\lambda/\!\!/\mu}(x \mid \beta) &\coloneqq \sum_{\nu \subseteq \mu} \beta^{|\mu/\nu|}GQ_{\lambda/\nu}(x \mid \beta)\label{GQ//},
\end{align}
where in both sums over all strict partitions $\nu \subseteq \mu$ such that $\mu/\nu \subseteq {\rm Rem}(\mu)$.
If $\lambda \nsupseteq \mu$, we put $GP_{\lambda/\!\!/\mu}(x \mid \beta)=0$ and $GQ_{\lambda/\!\!/\mu}(x \mid \beta)=0$.
These functions satisfy the following identities~\cite{CM23}:
\begin{align}
GP_{\lambda}(x,y \mid \beta)&= \sum_{\nu} GP_{\nu}(x \mid \beta)GP_{\lambda/\!\!/\nu}(y \mid \beta),\label{GP(x,y)}\\
GQ_{\lambda}(x,y \mid \beta)&= \sum_{\nu} GQ_{\nu}(x \mid \beta)GQ_{\lambda/\!\!/\nu}(y \mid \beta),\label{GQ(x,y)}
\end{align}
where the sums in both equations are over all strict partitions $\nu$.
These identities~\eqref{GP(x,y)} and~\eqref{GQ(x,y)} can also be derived from the definitions of $GP_{\lambda}(x \mid \beta)$ and $GQ_{\lambda}(x \mid \beta)$, respectively.
From Theorem~\ref{skewspecialvalue}, we obtain the following:
\begin{cor}\label{0skewspecialvalue}
Let $\lambda$ and $\mu$ be strict partitions.
For $\lambda \supsetneq \mu$, we have
\begin{align}
&GP_{\lambda/\!\!/\mu}(\beta,\dots,\beta \mid -\beta^{-1})=0,\label{ssGP}\\
&GQ_{\lambda/\!\!/\mu}(\beta,\dots,\beta \mid -\beta^{-1})=0\label{ssGQ}.
\end{align}
\end{cor}
\begin{proof}
We prove the equation~\eqref{ssGP}.
The equation~\eqref{ssGQ} is shown in a similar manner.
By~\eqref{GP//} and~\eqref{skewGPB}, we have
\begin{align}
GP_{\lambda/\!\!/\mu}(\beta,\dots,\beta \mid -\beta^{-1})&=\sum_{\nu \subseteq \mu}(-\beta^{-1})^{|\mu/\nu|}\beta^{|\lambda/\nu|}\notag\\
&=\sum_{\nu \subseteq \mu}(-1)^{|\mu|-|\nu|}\beta^{-(|\mu|-|\nu|)}\beta^{|\lambda|-|\nu|}\notag\\
&=\beta^{|\lambda|-|\mu|}\sum_{\nu \subseteq \mu}(-1)^{|\mu|-|\nu|}.\label{sgn}
\end{align}
We refer to a removable box as a corner of the shifted Young diagram.
The removable boxes of $\mu$ are located at the corners of the shifted Young diagram, since there are no adjacent boxes immediately to their right and below them. 
Thus, removing one corner does not affect the removability of the others. 
Therefore, the boxes in $\mathrm{Rem}(\mu)$ can be removed independently.
It follows that at most $|\mathrm{Rem}(\mu)| = a$ boxes can be removed from $\mu$.
Hence, the number of choices of $B \subseteq \text{Rem}(\mu)$ such that $|B|=b$ is $\binom{a}{b}$.
Furthermore, the box located at the bottom-rightmost corner of $\mu$ is removable, which implies that $a \neq 0$.
From these, we have the following:
\begin{align*}
\sum_{\nu \subseteq \mu}(-1)^{|\mu|-|\nu|}=\sum_{b=0}^{a}\binom{a}{b}(-1)^b=0.
\end{align*}
Consequently, the equation~\eqref{ssGP} is derived from~\eqref{sgn}.
\end{proof}
\begin{case}\label{exGP//}
We take $\lambda=(9,7,6,4) \in \mathcal{SP}_{26}$ and $\mu=(7,5,4,2) \in \mathcal{SP}_{18}$.
In this case, ${\rm Rem}(\mu)=\{(1,7),(3,6),(4,5)\}$.
That is, the boxes ${\rm Rem}(\mu)$ of $\mu$ are marked as follows:
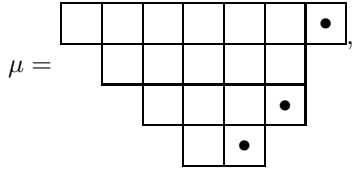
\vspace{5mm}
\begin{align*}
{\raisebox{-11pt}[0pt][0pt]{$\mu=\ $}}{\raisebox{0pt}[0pt][0pt]{\ytableaushort{\ \ \ \ \ \ {\bullet},\none\ \ \ \ \ ,\none\none\ \ \ {\bullet},\none\none\none\ {\bullet}}},}\\[15mm]
\end{align*}
where the union of boxes with black dots represents ${\rm Rem}(\mu)$.
Enumerating all $\nu$ that appear the sum~\eqref{GP//}, we have the following:
\vspace{3mm}
\begin{align*}
&{\raisebox{-11pt}[0pt][0pt]{${\nu}:\ $}}{\raisebox{0pt}[0pt][0pt]{\ytableaushort{\ \ \ \ \ \ {\bullet},\none\ \ \ \ \ ,\none\none\ \ \ {\bullet},\none\none\none\ {\bullet}}},\ }{\raisebox{0pt}[0pt][0pt]{\ytableaushort{\ \ \ \ \ \ {\bullet},\none\ \ \ \ \ ,\none\none\ \ \ {\bullet},\none\none\none\ }},\ \ }{\raisebox{0pt}[0pt][0pt]{\ytableaushort{\ \ \ \ \ \ ,\none\ \ \ \ \ ,\none\none\ \ \ {\bullet},\none\none\none\ {\bullet}}},}\\[25mm]
&\phantom{{\raisebox{-11pt}[0pt][0pt]{${\nu}:\ $}}}{\raisebox{0pt}[0pt][0pt]{\ytableaushort{\ \ \ \ \ \ {\bullet},\none\ \ \ \ \ ,\none\none\ \ \ ,\none\none\none\ {\bullet}}},\ \ }{\raisebox{0pt}[0pt][0pt]{\ytableaushort{\ \ \ \ \ \ {\bullet},\none\ \ \ \ \ ,\none\none\ \ \ ,\none\none\none\ }},\ }{\raisebox{0pt}[0pt][0pt]{\ytableaushort{\ \ \ \ \ \ ,\none\ \ \ \ \ ,\none\none\ \ \ {\bullet},\none\none\none\ }},}\\[25mm]
&\phantom{{\raisebox{-11pt}[0pt][0pt]{${\nu}:\ $}}}{\raisebox{0pt}[0pt][0pt]{\ytableaushort{\ \ \ \ \ \ ,\none\ \ \ \ \ ,\none\none\ \ \ ,\none\none\none\ {\bullet}}},\ \ }{\raisebox{0pt}[0pt][0pt]{\ytableaushort{\ \ \ \ \ \ ,\none\ \ \ \ \ ,\none\none\ \ \ ,\none\none\none\ }}.}\\[15mm]
\end{align*}
Thus, it follows that
\begin{align*}
GP_{\lambda/\!\!/\mu}(\beta,\dots,\beta \mid -\beta^{-1})&=
\beta^{26-18}\sum_{\nu \subseteq \mu}(-1)^{18-|\nu|}\\
&=\beta^8\sum_{b=0}^3 \binom{3}{b}(-1)^b\\
&=\beta^{8}\{ (-1)^{0}+3(-1)^{1}+3(-1)^{2}+(-1)^3 \}=0.
\end{align*}
\end{case}
By the equation~\eqref{GP//}, we have
\begin{align}
GP_{\lambda/\!\!/\mu}(1,\dots,1 \mid -1) &\equiv GP_{\lambda/\!\!/\mu}(1,\dots,1 \mid 1)\pmod 2\notag \\[-1mm]
&=\sum_{\nu \subseteq \mu} |{\rm SSVT}_P(\lambda/\nu,n)|\label{even}.
\end{align}
From the equation~\eqref{ssGP} in Corollary~\ref{0skewspecialvalue}, we have $GP_{\lambda/\!\!/\mu}(1,\dots,1 \mid -1)=0$.
Thus, the value of the right-hand side of~\eqref{even} is even.
Therefore, we can construct a pairing correspondence between the shifted set-valued skew tableaux that appear in~\eqref{even}.
We explicitly construct pairs of tableaux in the following.

For $\emptyset \subseteq B \subseteq {\rm Rem}(\mu)$, we put $\nu^B$ such that $\mu/\nu^B=B$.
Let $N = \{ \nu^{B} \mid \emptyset \subseteq B \subseteq {\rm Rem}(\mu)\}$, and $b^\mu$ be the box located at the bottom-right corner of $\mu$.
Note that the box $b^\mu$ belongs to the set of removable boxes, $b^\mu \in \text{Rem}(\mu)$.

We define a map $\pi$ on the set $N$ as follows:
\begin{align*}\label{pi}
  \begin{split}
    &\pi : N \rightarrow N \\
    &\nu \mapsto
  \begin{cases}
    \nu \setminus \{ b^\mu \} & (b^\mu \in \nu),\\
    \nu \sqcup \{ b^\mu \} &  (b^\mu \notin \nu).
  \end{cases}
  \end{split}
\end{align*}
We show that the map $\pi$ is well-defined.
Let $\nu$ be a strict partition in $N$.
If $b^\mu \in \nu$, then $\nu$ corresponds to a subset $B_1 \subseteq \text{Rem}(\mu)$ such that $b^\mu \notin B_1$.
Since $b^\mu$ is removable, $\nu \setminus \{b^\mu\}$ is again a strict partition, which we denote as $\nu^{B_1 \sqcup \{b^\mu\}}$.
Thus, it holds that
\begin{align*}
\pi(\nu) = \nu \setminus \{ b^\mu \} = \nu^{B_1 \sqcup \{ b^\mu \}} \in N.
\end{align*}
Conversely, if $b^\mu \notin \nu$, then $\nu$ is associated with a subset $B_2 \subseteq \text{Rem}(\mu)$ such that $b^\mu \in B_2$.
Given that $b^\mu$ is removable, $\nu \sqcup \{b^\mu\}$ yields a strict partition, written as $\nu^{B_2 \setminus \{b^\mu\}}$.
We therefore obtain
\begin{align*}
\pi(\nu) = \nu \sqcup \{ b^\mu \} = \nu^{B_2 \setminus \{ b^\mu \}} \in N.
\end{align*}
Consequently, for every partition $\nu \in N$, the image $\pi(\nu)$ is also a strict partition contained in $N$.
Thus, the map $\pi$ is well-defined.

It immediately follows that the map $\pi$ is an involution on $N$.
Indeed, if $b^\mu \in \pi(\nu)$, we find that
\begin{align*}
\pi(\pi(\nu)) = \pi(\nu \setminus \{ b^\mu \}) = (\nu \setminus \{ b^\mu \}) \sqcup \{ b^\mu \} = \nu.
\end{align*}
On the other hand, if $b^\mu \notin \pi(\nu)$, we have
\begin{align*}
\pi(\pi(\nu)) = \pi(\nu \sqcup \{ b^\mu \}) = (\nu \sqcup \{ b^\mu \}) \setminus \{ b^\mu \} = \nu.
\end{align*}
Therefore, since $\pi(\pi(\nu)) = \nu$ for all $\nu \in N$, the map $\pi$ is an involution on $N$.

Next, we verify the one-to-one correspondence between the shifted skew tableaux appearing in the equation~\eqref{GP//}.
By Lemma~\ref{lem1}, there exists a unique shifted skew tableau for any skew Young diagram satisfying the condition~\eqref{dfnT_P}.
Therefore, based on the definition of $\pi$ and Lemma~\ref{lem1}, the tableaux $T^{\lambda/\nu}_P \in {\rm ShT}_P(\lambda/\nu,n)$ and $T^{\lambda/\pi(\nu)}_P \in {\rm ShT}_P(\lambda/\pi(\nu),n)$ form a one-to-one pair.
The remaining tableaux in the set ${\rm SSVT}'_P(\lambda/\nu,n)$ are paired by the sign-reversing involution $\iota_P$.
Hence, we conclude that the shifted set-valued skew tableaux in the summation~\eqref{GP//} are paired by $\pi$ and $\iota_P$. 
The same argument applies to the summation~\eqref{GQ//} using the involutions $\pi$ and $\iota_Q$.
In summary, we obtain the following result.
\begin{prop}\label{pair}
For the strict partitions $\lambda \supsetneq \mu$ and $\nu$ that satisfy $\mu/\nu \subseteq {\rm Rem}(\mu)$, the shifted set-valued skew tableaux appearing in the sums~\eqref{GP//} and~\eqref{GQ//} form pairs by either $\iota_P$ or $\iota_Q$ and $\pi$.
\end{prop}

\begin{case}
We consider the same case as in Example~\ref{exGP//}.
In this situation, the box $b^\mu$ is located at $(4,5)$.
We consider the pair of corresponding shifted skew tableaux $T^{\lambda/\nu}_P \in {\rm ShT}(\lambda/\nu,n)$ and $T^{\lambda/\pi(\nu)}_P \in {\rm ShT}(\lambda/\pi(\nu),n)$ determined by the involution $\pi$.
The corresponding pairs are connected by arrows as shown below:
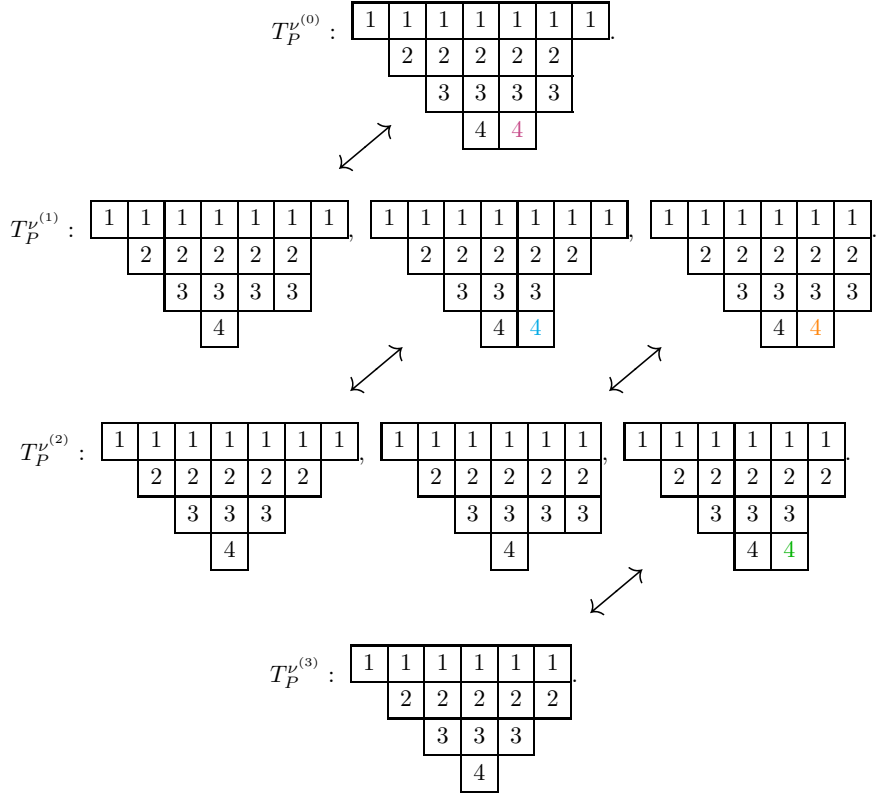
\vspace{5mm}
\begin{align*}
{\scalebox{0.9}{{\raisebox{10pt}[0pt][0pt]{\ytableaushort{\none\none\none\none\none\none{\ 1'}1,\none\none\none\none\none{\ 1'}1,\none\none\none\none\none{1}{\ 2'},\none\none\none\none{\ 1'}{\ 2'}}.}}}}\\[5mm]
\scalebox{1.5}{\rotatebox{220}{$\longleftrightarrow$}}\ \ \ \ \ \ \ \ \ \ \ \ \ \ \ \ \ \ \ \ 
\end{align*}
\begin{align*}
{\scalebox{0.9}{{\raisebox{10pt}[0pt][0pt]{\ytableaushort{\none\none\none\none\none\none{\ 1'}1,\none\none\none\none\none{\ 1'}1,\none\none\none\none\none{1}{\ 2'},\none\none\none{\color{magenta!80!black}\ 1'}{1}{\ 2'}},}}}}
{\scalebox{0.9}{{\raisebox{10pt}[0pt][0pt]{\ytableaushort{\none\none\none\none\none{\ 1'}1,\none\none\none\none{\ 1'}1,\none\none\none{\ 1'}1{\ 2'},\none\none\none1{\ 2'}},\ \ \ \ \ }}}}
{\scalebox{0.9}{{\raisebox{10pt}[0pt][0pt]{\ytableaushort{\none\none\none{\ 1'}1{1},\none\none\none1{\ 2'},\none\none\none{\ 2'}2,\none\none{\ 1'}2}.}}}}\\[8mm]
\scalebox{1.5}{\rotatebox{220}{$\longleftrightarrow$}}\ \ \ \ \ \ \ \ \ \ \ \ \ \ \ \ \ \ \ \ \ \  {\scalebox{1.5}{\rotatebox{220}{$\longleftrightarrow$}}}\ \ \ \ \ \ \ \ \ \ \ \ \ \ \ \ \ \ \ \ \ \ 
\end{align*}
\begin{align*}
{\scalebox{0.9}{{\raisebox{10pt}[0pt][0pt]{\ytableaushort{\none\none\none\none\none{\ 1'}1,\none\none\none\none{\ 1'}1,\none\none\none{\ 1'}1{\ 2'},\none\none{\color{cyan!95!black}\ 1'}1{\ 2'}},\ \ \ \ \ }}}}
{\scalebox{0.9}{{\raisebox{10pt}[0pt][0pt]{\ytableaushort{\none\none\none{\ 1'}1{1},\none\none\none1{\ 2'},\none\none\none{\ 2'}2,\none{\color{orange!95!white}\ 1'}12},}}}}
{\scalebox{0.9}{{\raisebox{10pt}[0pt][0pt]{\ytableaushort{\none\none\none\none{\ 1'}1{1},\none\none\none\none1{\ 2'},\none\none\none{\ 1'}{\ 2'}{2},\none\none\none12}.\ \ }}}}\\[8mm]
\scalebox{1.5}{\rotatebox{220}{$\longleftrightarrow$}}\ \ \ \ \ \ \ \ \ \ \ \ \ \ \ \ \ \ \ \ \ \ \ 
\end{align*}
\begin{align*}
{\scalebox{0.9}{{\raisebox{10pt}[0pt][0pt]{\ytableaushort{\none\none\none\none{\ 1'}1{1},\none\none\none\none1{\ 2'},\none\none\none{\ 1'}{\ 2'}{2},\none\none{\color{green!70!black}\ 1'}12}.\ \ }}}}
\end{align*}\vspace{8mm}

\noindent
The arrows, which are drawn with heads at both ends, indicate the one-to-one correspondence between the shifted skew tableaux determined by the involution $\pi$.
The boxes containing a single colored number indicate the position of the box $b^\mu$, which is the bottom-leftmost box in either $\lambda/\nu$ or $\lambda/\pi(\nu)$.

\end{case}

\section{Conclusions}\label{section5}

In this paper, we showed special values of $K$-theoretic Schur $P$- and $Q$-functions $GP_{\lambda/\mu}(x \mid \beta)$ and $GQ_{\lambda/\mu}(x \mid \beta)$ (Theorem~\ref{skewspecialvalue}).
This result was shown by defining sign-reversing involutions $\iota_P$ and $\iota_Q$ on the sets ${\rm SSVT}'_P(\lambda/\mu,n)$ and ${\rm SSVT}'_Q(\lambda/\mu,n)$, respectively.
These involutions are maps that form pairs of tableaux, except for specific tableaux $T_P$ in ${\rm SSVT}_P(\lambda/\mu,n)$ and $T_Q$ in ${\rm SSVT}_Q(\lambda/\mu,n)$.
Furthermore, these possess the property that the paired tableaux obtained from the specialization of variables and parameters have opposite signs (see Remark~\ref{remsri}).
This property implies that both special values $GP_\lambda(\beta,\dots,\beta \mid -\beta^{-1} )$ and $GQ_\lambda(\beta,\dots,\beta \mid -\beta^{-1})$ become the monomial $\beta^{|\lambda|}$.
Using this special value, we gave the odd-number property of the number of shifted set-valued tableaux  ${\rm SSVT}_P(\lambda/\mu,n)$ and ${\rm SSVT}_Q(\lambda/\mu,n)$ (Corollary~\ref{oddness}).

Additionally, there exist functions defined using $GP_{\lambda/\mu}(x \mid \beta)$ and $GQ_{\lambda/\mu}(x \mid \beta)$, which are denoted as $GP_{\lambda/\!\!/\mu}(x \mid \beta)$ and $GQ_{\lambda/\!\!/\mu}(x \mid \beta)$, respectively.
When the specialization of variables and the substitution of the parameter are performed in the same manner, the special values $GP_{\lambda/\!\!/\mu}(\beta,\dots,\beta \mid -\beta^{-1})$ and $GQ_{\lambda/\!\!/\mu}(\beta,\dots,\beta \mid -\beta^{-1})$ both become 0 (Corollary~\ref{0skewspecialvalue}).
Using this special value, we showed that the shifted set-valued skew tableaux appearing in    the definition of $GP_{\lambda/\!\!/\mu}(x \mid \beta)$ or $GQ_{\lambda/\!\!/\mu}(x \mid \beta)$ form pairs through the map $\pi$ and the sign-reversing involution $\iota_P$ or $\iota_Q$ (Proposition~\ref{pair}).

Counting the number of set-valued tableaux is an interesting problem.
In~\cite{FNS24}, an explicit formula for the number of set-valued tableaux was given.
This formula was obtained from the bi-alternant formula of Grothendieck polynomials.
Furthermore, it was mentioned in~\cite{FNS24} that the number of set-valued semistandard tableaux is related to the Gauss hypergeometric function and the Holman hypergeometric function.
If a similar bi-alternant formula is provided for $GP_{\lambda}(x \mid \beta)$ and $GQ_{\lambda}(x \mid \beta)$ polynomials, it is expected that an explicit formula for the number of shifted set-valued tableaux will be obtained.
It is also a significant problem to explore relationships between those explicit formulas and hypergeometric functions.

\section*{Acknowledgments} 
The authors would like to express their sincere gratitude to Yasuhiko Yamada for helpful discussions and constant encouragement.
The authors are deeply grateful to Yusuke Nakayama and Shogo Sugimoto for informing them about $K$-theoretic Schur $P$- and $Q$-functions, introducing relevant references in the early stages of the research, contributing to engaging discussions, and providing valuable comments on the preprint.
The authors extend their deepest thanks to Travis Scrimshaw for reviewing the preprint and offering his insightful comments on Corollary~\ref{0skewspecialvalue}.
This work was supported by JST SPRING, Grant Number JPMJSP2148, and JSPS KAKENHI Grant Number 22H01116.

\noindent{\sc Department of Education, Kogakkan University}

(Takahiko Nobukawa) {\it E-mail address}:  {\tt t-nobukawa@kogakkan-u.ac.jp}

\noindent{\sc Department of Mathematics, Graduate School of Science, Kobe University}

(Tatsushi Shimazaki) {\it E-mail address}: {\tt tsimazak@math.kobe-u.ac.jp}


\begin{thebibliography}{LS82}
\bibitem[1]{Buc02} A. S. Buch, \textit{A Littlewood-Richardson rule for the K-theory of Grassmannians}, Acta. Math.,
{\bf 189}(1): 37$\mathchar`-$78, 2002.
\bibitem[2]{CM23} Y.-C. Chiu and E. Marberg, \textit{Expanding K-theoretic Schur $Q$-functions}, Algebr. Comb., {\bf 6}(6): 1419–1445, 2023.
\bibitem[3]{FNS23} T. Fujii, T. Nobukawa and T. Shimazaki, \textit{The number of the set-valued tableaux is odd}, arXiv: 2305.06740, 2023.
\bibitem[4]{FNS24} T. Fujii, T. Nobukawa and T. Shimazaki, \textit{Special values of Grothendieck polynomials in terms of hypergeometric functions}, Hiroshima Mathematical Journal, {\bf 55}(2): 167-182, 2025.
\bibitem[5]{IN13} T. Ikeda and H. Naruse, \textit{K-theoretic analogues of factorial Schur P-and Q-functions}, In: Adv. Math., {\bf 243}: 22$\mathchar`-$66, 2013.
\bibitem[6]{Iva01} V. N. Ivanov, \textit{Combinatorial formula for factorial Schur Q-functions}, J. Math. Sci. (N.Y.) {\bf 107}, 4195–4211, 2001.
\bibitem[7]{Iva05} V. N. Ivanov, \textit{Interpolation analogues of Schur Q-functions}, J. Math. Sci. (N.Y.) {\bf 131}, 5495–5507, 2005.
\bibitem[8]{Las90} A. Lascoux, \textit{Anneau de Grothendieck de la vari\'{e}t\'{e} de drapeaux}, The Grothendieck Festschrift, Vol. {\bf III}: Progr. Math., Birkh\"{a}user, Boston, 1$\mathchar`-$34, 1990.
\bibitem[9]{LS82} A. Lascoux and M. P. Sch\"{u}tzenberger, \textit{Structure de Hopf de l'anneau de cohomologie et de
l'anneau de Grothendieck d'une vari\'{e}t\'{e} de drapeaux}, C. R. Acad. Sci. Paris S\'{e}r. I Math, {\bf 295}(11): 629$\mathchar`-$633,
1982.
\bibitem[10]{LM21} J. B. Lewis, E. Marberg, \textit{Enriched set-valued P-partitions and shifted stable Grothendieck polynomials}, Math. Z. {\bf 299}, 1929–1972, 2021.
\bibitem[11]{Mac95} I. G. Macdonald, \textit{Symmetric Functions and Hall Polynomials}, 2nd. ed., Oxford, 1995.
\bibitem[12]{M24} E. Marberg, \textit{Shifted combinatorial Hopf algebras from $K$-theory}, Algebr. Comb., {\bf 7}(4): 1123–1156, 2024.
\bibitem[13]{Sch11} I. Schur, \textit{\"{U}ber die Darstellung der symmetrischen und der alternierenden Gruppe durch gebrochene lineare Substitutionen}, J. Reine Angew. Math. {\bf 139}, 155–250, 1911.
\bibitem[14]{Ste89} J. R. Stembridge, \textit{Shifted tableaux and the projective representations of symmetric groups}, Adv.
Math. {\bf 74}(1), 87–134, 1989.
\end{thebibliography}
\end{document}